\documentclass[11pt, a4paper,leqno]{amsart}
\usepackage{amsmath,amsthm,amscd,amssymb,amsfonts, amsbsy}
\usepackage{latexsym}
\usepackage{txfonts}
\usepackage{exscale}
\usepackage{centernot}
\usepackage[colorlinks,citecolor=red,pagebackref,hypertexnames=false]{hyperref}
\usepackage{color}

\calclayout
\allowdisplaybreaks

\theoremstyle{plain}
\newtheorem{theorem}[equation]{Theorem}
\newtheorem{lemma}[equation]{Lemma}
\newtheorem{corollary}[equation]{Corollary}

\theoremstyle{definition}
\newtheorem{definition}[equation]{Definition}

\theoremstyle{remark}
\newtheorem{remark}[equation]{Remark}

\numberwithin{equation}{section}

\newcommand{\eps}{\varepsilon}

\newcommand{\dist}{\operatorname{dist}}

\newcommand{\re}{\mathbb{R}}

\newcommand{\ree}{\mathbb{R}^{n+1}}

\newcommand{\om}{\Omega}

\newcommand{\nt}{n.t. \nabla\mathcal{S}}

\newcommand{\pom}{\partial\Omega}

\newcommand{\hm}{\omega}

\newcommand{\RNum}[1]{\uppercase\expandafter{\romannumeral #1\relax}}

\renewcommand{\emptyset}{\mbox{\textup{\O}}}

\DeclareMathOperator{\supp}{supp}

\DeclareMathOperator{\diam}{diam}

\DeclareMathOperator{\interior}{int}

\begin{document}
\allowdisplaybreaks

\title[A Two Phase Free  Boundary Problem]{A Singular Integral approach to a 
Two Phase Free  Boundary Problem}

\author{Simon Bortz}

\address{Simon Bortz
\\
Department of Mathematics
\\
University of Missouri
\\
Columbia, MO 65211, USA} \email{sabh8f@mail.missouri.edu}

\author{Steve Hofmann}

\address{Steve Hofmann
\\
Department of Mathematics
\\
University of Missouri
\\
Columbia, MO 65211, USA} \email{hofmanns@missouri.edu}

\thanks{The authors were supported by NSF grant DMS-1361701.}

\date{\today}


\begin{abstract} We present an alternative proof of a result of Kenig and Toro \cite{KT4}, which states 
that if $\om\subset \ree$ is a two sided NTA domain, with Ahlfors-David regular boundary, 
and the $\log$ of the Poisson kernel associated to 
$\om$ as well as the $\log$ of the Poisson kernel associated to ${\Omega_{\rm ext}}$ are in 
VMO, then the outer unit normal $\nu$ is in VMO. 
Our proof exploits the usual jump relation formula for the non-tangential limit of the gradient of the single 
layer potential.  We are also able to relax the assumptions of Kenig and Toro in the case that the pole for the 
Poisson kernel is finite:  in this case, we assume only that $\pom$ is uniformly rectifiable, and that $\pom$
coincides with  the measure theoretic boundary of $\om$ a.e. with respect to Hausdorff $H^n$ measure.

\end{abstract}

\maketitle


\section{Introduction}

The study of free boundary problems for Poisson kernels, in which
 the regularity of the Poisson kernel $k$ associated to a domain $\om$, is shown to imply regularity of its 
boundary $\pom$, began with the work of Alt and Caffarelli \cite{AC}. In \cite{AC} it is shown that in the 
presence of sufficient Reifenberg flatness, plus Ahlfors-David regularity of the boundary, 
$\log k \in C^\alpha \implies \nu \in C^\beta$, where $\nu$ is the unit normal to $\pom$. In \cite{J}, 
Jerison showed that $\beta$ can be taken to be equal to $\alpha$.  In the same work, 
Jerison studied the regularity of the unit normal under the assumption that $\log k$ was 
continuous and that $\pom$ was Lipschitz.   The converse to the Alt-Caffarelli-Jerison $C^\alpha$ result
is a classical theorem of Kellog \cite{Kel}.

Kenig and Toro investigated what may be considered the `end point' case of the results in \cite{AC}, 
again under the assumption of Reifenberg flatness and Ahlfors-David regularity. In \cite{KT1} and \cite{KT3}, 
they showed that given those background hypotheses, one obtains
$$\log k \in VMO(d \sigma) \iff \nu \in VMO(d \sigma)$$
(more precisely, for the direction ``$\implies$", they show that $\nu \in VMO_{loc}(d\sigma)$;
we shall return to this point below),
where for a domain $\om\subset \ree$, we let $\sigma:= H^n|_{\pom}$ denote the surface measure on
$\pom$, where as usual $H^n$ denotes $n$-dimensional Hausdorff measure.
It is well known that in the absence of Reifenberg flatness, 
$\log k \in VMO(d \sigma)$ 
need not imply that $\nu \in VMO(d \sigma)$: the counterexample is the double cone (see \cite{AC} and \cite{KP}). 
With this in mind one may ask under what other hypotheses do we have 
$\log k \in VMO(d \sigma) \implies \nu \in VMO(d \sigma)$, 
and when do we have $\log k \in VMO(d \sigma) \iff \nu \in VMO(d \sigma)$?

The second author, along with M. Mitrea and M. Taylor \cite{HMT}, showed that if $\om$ is 2-sided NTA, with 
Ahlfors-David regular (``ADR") boundary, 
then $\nu \in VMO(d \sigma)$ implies vanishing Reifenberg flatness, and hence that $ \log k \in VMO(d \sigma)$,
by the result of \cite{KT1}. Thus, it seems reasonable to conjecture that  
in the setting of a 2-sided NTA domain with ADR boundary, 
$\log k \in VMO(d \sigma) \implies \nu \in VMO(d \sigma)$. 
Indeed, this very question was posed in \cite{KT4}. In trying to answer this question 
(which remains open), the authors of the present work
found an alternative approach, based on 
$L^p$ bounds and jump relations for
layer potentials, to prove a two-phase version of the problem, which had previously been treated in \cite{KT4}. 

In \cite{KT4} Kenig and Toro established the following: suppose that $\om$ is a 2-sided chord-arc domain
(i.e, a 2-sided NTA domain with ADR boundary), and 
that $k_1$ and $k_2$ are the Poisson kernels of 
$\om$ and ${\Omega_{\rm ext}}$ respectively, with some fixed poles (either finite or infinite).   
If $\log k_1, \log k_2 \in VMO(d \sigma)$ (or even $VMO_{loc}$), then $\nu \in VMO_{loc}(d \sigma)$. 
Their proof uses a blow-up argument which is quite natural given the nature of the problem. 
This blow-up argument, as well as our arguments here, require that one work with local versions of $VMO$;
see Definition \ref{defvmo} below.
In the 
present paper, we will use $L^p$ bounds and jump relations for 
the gradient of the single layer potential  to show that if $\log k_1, \log k_2 \in VMO_{loc}(d \sigma)$, 
then $\nu \in VMO_{loc}(d \sigma)$. 
Our approach has the additional modest virtue that it allows us to weaken the 2-sided NTA condition
in a non-trivial way,
in the case that our Poisson kernels have finite poles.  More precisely, 
our main result is the following.  The terminology used in the theorem and throughout this introduction will be defined
in the sequel.  
\begin{theorem}\label{T1}
Let $\om$ and  ${\Omega_{\rm ext}}$ be connected domains in $\ree$,
whose common boundary $\pom = \pom_{ext}$ is uniformly rectifiable, and whose measure theoretic boundary
$\partial_*\om$ satisfies $\sigma(\pom \setminus  \partial_*\om)=0$.\footnote{Following \cite{HMT}, 
we use the term ``UR domain", to refer to a domain whose boundary enjoys these properties; see
Definition \ref{URdom} below.} 
Suppose that there is some fixed pair of points  $X_1 \in \om$ and $ X_2\in {\Omega_{\rm ext}}$, such that
$\log k_1^{X_1}, \log k_2^{X_2} \in VMO_{loc}( d\sigma)$, 
where $k_1$ and $k_2$ are the Poisson kernels for $\om$ and ${\Omega_{\rm ext}}$ respectively.
Then $\nu \in VMO_{loc}(d\sigma)$.
\end{theorem}
Thus, uniform rectifiability of the boundary, along with the hypothesis that the measure theoretic boundary
has full measure, replace the stronger 2-sided chord arc condition\footnote{It is known that
chord arc domains have uniformly rectifiable boundaries \cite{DJe}, and in fact, the chord arc condition 
is strictly stronger.}, 
at least in the case that our Poisson kernels 
have finite poles.  On the other hand, in the case that the Poisson kernel has pole at infinity,
we impose an NTA hypothesis, as in \cite{KT4}.  Our second result is as follows.
\begin{theorem}\label{T2} \
\begin{enumerate}
\item Suppose that $\om$ is an unbounded chord arc domain with $\om_{ext}$ connected,  
that $\log k_1 \in VMO(d \sigma)$, where $k_1$ is the Poisson kernel for $\om$ with pole at infinity, 
and that $\log k_2^{X_2} \in VMO_{loc}(d \sigma)$, where $k_2^{X_2}$ is the Poisson kernel for 
${\Omega_{\rm ext}}$ with pole ${X_2} \in {\Omega_{\rm ext}}$.  Then $\nu \in VMO_{loc}(d \sigma)$.

\item Suppose $\om$ is an unbounded 2-sided chord-arc domain, with unbounded complement, 
and that $\log k_1, \log k_2 \in VMO(d \sigma)$, where $k_1$ is the 
Poisson kernel for $\om$ with pole at infinity, and $k_2$ is the 
Poisson kernel for ${\Omega_{\rm ext}}$ with pole at infinity.   Then $\nu \in VMO(d \sigma)$.
\end{enumerate}
\end{theorem}

As an immediate consequence, we obtain the following, as in \cite{KT4}.
\begin{corollary}\label{c1}  Suppose that $\om$ is a 2-sided chord-arc domain, and let $k_1$, $k_2$
be the Poisson kernels for $\om$ and for $\om_{ext}$, respectively, with fixed poles which may either be
finite or infinite.  Suppose also that
$\log k_1 \in C^\alpha$, and that $\log k_2 \in VMO_{loc}(d \sigma)$.
Then $\pom$ is locally the graph of a $C^{1,\alpha}$ function. 
\end{corollary}

Let us indicate the proof the the Corollary in the case that $\pom$ is bounded
(then $VMO_{loc} = VMO$ in this case). Observe that 
we may invoke Theorem \ref{T1} or Theorem \ref{T2} to deduce that $\nu \in VMO$.  Consequently, by the aforementioned result of \cite{HMT}, we find that $\Omega$ is a vanishing Reifenberg domain. Thus, by the 
Alt-Caffarelli-Jerison theorem (\cite{AC}, \cite{J}), 
$\nu\in C^\alpha(\pom)$. If $\pom$ is unbounded, we note that the results of  \cite{AC} and \cite{J} are 
local in nature, and furthermore observe that the result in \cite{HMT} may be localized as well, with 
essentially the same proof.

Throughout the present paper $\om$ will be a connected open (proper) subset of $\ree$. We begin by 
setting notation, and recalling some definitions. 

\subsection{Notation and Definitions}

\begin{list}{$\bullet$}{\leftmargin=0.4cm  \itemsep=0.2cm}

\item We use the letters $c,C$ to denote harmless positive constants, not necessarily
the same at each occurrence, which depend only on dimension and the
constants appearing in the hypotheses of the theorems (which we refer to as the
``allowable parameters'').  We shall also
sometimes write $a\lesssim b$ and $a \approx b$ to mean, respectively,
that $a \leq C b$ and $0< c \leq a/b\leq C$, where the constants $c$ and $C$ are as above, unless
explicitly noted to the contrary.

\item Unless otherwise explicitly stated we will use lower case letters $x,y,z$, etc., to denote points on $\pom$, and capital letters
$X,Y,Z$, etc., to denote generic points in $\ree$ (especially those in $\om$ and $\Omega_{ext}$). 

\item The open $(n+1)$-dimensional Euclidean ball of radius $r$ will be denoted
$B(x,r)$ when the center $x$ lies on $\pom$, or $B(X,r)$ when the center
$X \in (\pom)^c$.  A ``surface ball'' is denoted
$\Delta(x,r):= B(x,r) \cap\partial\Omega.$



\item Given a domain $\om$, for $X \in \ree$, we set $\delta(X):= \dist(X,\pom)$. 

\item We let $H^n$ denote $n$-dimensional Hausdorff measure, and let
$\sigma := H^n\big|_{E}$ denote the ``surface measure'' on a closed set $E\subset \ree$
of co-dimension 1.

\item For a Borel set $A\subset \ree$, we let $1_A$ denote the usual
indicator function of $A$, i.e. $1_A(x) = 1$ if $x\in A$, and $1_A(x)= 0$ if $x\notin A$.

\item For a Borel set $A\subset \ree$,  we let $\interior(A)$ denote the interior of $A$.


\item Given a Borel measure $\mu$, and a Borel set $A$, with positive and finite $\mu$ measure, we
set $\fint_A f d\mu := \mu(A)^{-1} \int_A f d\mu$.

\end{list}

\begin{definition}\label{defadr} ({\bf  ADR})  (aka {\it Ahlfors-David regular}).
We say that a  set $E \subset \ree$, of Hausdorff dimension $n$, is ADR
if it is closed, and if there is some uniform constant $C$ such that
\begin{equation} \label{eq1.ADR}
\frac1C\, r^n \leq \sigma\big(\Delta(x,r)\big)
\leq C\, r^n,\quad\forall r\in(0,\diam (E)),\ x \in E,
\end{equation}
where $\diam(E)$ may be infinite.
\end{definition}

\begin{definition}({\bf Riesz transforms and the single layer potential})
Let $E\subset \ree$ be an $n$-dimensional ADR (hence closed) set with surface measure $\sigma$.
We define the (vector valued) Riesz kernel as
\begin{equation}\label{RieszKern}
\mathcal{K}(x) = \tilde{c}_n\frac{x}{|x|^{n+1}}
\end{equation}
where $\tilde{c}_n$ is chosen so that $\mathcal{K}$ is the gradient of fundamental solution to the Laplacian. 
For a Borel measurable function $f$, we then define the Riesz transform 
\begin{equation}\label{RieszXform}
\mathcal{R}f(X) := \mathcal{K} \ast (f\sigma)(X) = \int_{E} \mathcal{K}(X-y)f(y)\, d\sigma(y) \quad X \in \ree\,,
\end{equation} 
as well as the truncated Riesz transforms
$$\mathcal{R}_\eps f(X):= \int_{E\,\cap\,\{|X-y|>\eps\}} \mathcal{K}(X-y)f(y)\, d\sigma(y)\,,\qquad  \eps>0\,.$$
We define $\mathcal{S}$ the single layer potential for the Laplacian relative to $E$ to be
\begin{equation}
\mathcal{S}f(X): = \int_{E} \mathcal{E}(X -y) f(y) \, d\sigma(y),
\end{equation}
where $\mathcal{E}(X) = c_n|X|^{1-n}$ is the (positive) fundamental solution to the 
Laplacian in $\ree$. Notice that $\nabla \mathcal{S}f(X) = \mathcal{R}f(X)$ for $X \not\in E$.
\end{definition}

\begin{definition}\label{defur} ({\bf UR}) (aka {\it uniformly rectifiable}).
There are numerous characterizations of UR sets (many of which remain valid in higher co-dimensions); 
we refer the interested reader to \cite{DS1,DS2} for details.  For our purposes, it is most useful to use the following definition.
Let $E\subset \ree$ be an $n$-dimensional ADR (hence closed) set with surface measure $\sigma$. Then $E$ is UR if and only if the Riesz transform operator is $L^2$ bounded with respect to surface measure, in the sense that
 \begin{equation}\label{eqrtbound}
 \sup_{\eps>0} \|\mathcal{R}_\eps f\|_{L^2(E,\sigma)} \le C\rVert f \rVert_{L^2(E,\sigma)}\,.
 \end{equation}
 See \cite{DS1} for a proof of the fact that UR (defined in various other ways) implies
 \eqref{eqrtbound}.  For the converse, see \cite{MMV} in the case $n=1$, and \cite{NToV} in general.
 For our purposes, we shall require only the first implication (UR implies \eqref{eqrtbound}).
\end{definition}

\begin{remark}\label{r1}
We note that the principal value $\mathcal{R} f(x)= \lim_{\eps\to 0} \mathcal{R}_\eps f(x)$
exists for a.e. $x\in E$, provided that $E$ is rectifiable, thus, in particular, if $E$ is UR 
(see, e.g., \cite[Theorem 20.28]{M}).
Of course, the $L^2$ bound in \eqref{eqrtbound} holds also for the principal value operator.
Moreover, by standard Calder\'on-Zygmund theory, the $L^2$ bound self-improves to give $L^p$ bounds, $1<p<\infty$.
\end{remark}

\begin{definition}\label{defurchar} ({\bf ``UR character"}).   Given a UR set $E\subset \ree$, its ``UR character"
is just the constant $C$ in \eqref{eqrtbound},
along with the ADR constant; or equivalently,
the quantitative bounds involved in any particular characterization of uniform rectifiability.
\end{definition}

\begin{definition}\label{URdom}({\bf UR domain}). Following the terminology in \cite{HMT},
we will say that a domain $\om$ is a UR domain if $\pom$ is UR,
and if the measure theoretic boundary
$\partial_*\om$ (see Definition \ref{MTbdry} below)
satisfies $\sigma(\pom \setminus  \partial_*\om)=0$.
\end{definition}

\begin{definition} ({\bf Corkscrew condition}).  \label{def1.cork}
Following
\cite{JK}, we say that a domain $\Omega\subset \ree$
satisfies the ``Corkscrew condition'' if for some uniform constant $c>0$ and
for every surface ball $\Delta:=\Delta(x,r),$ with $x\in \partial\Omega$ and
$0<r<\diam(\partial\Omega)$, there is a ball
$B(X_\Delta,cr)\subset B(x,r)\cap\Omega$.  The point $X_\Delta\subset \Omega$ is called
a ``Corkscrew point'' relative to $\Delta.$  We note that  we may allow
$r<C\diam(\pom)$ for any fixed $C$, simply by adjusting the constant $c$.
\end{definition}

\begin{definition}({\bf Harnack Chain condition}).  \label{def1.hc} Again following \cite{JK}, we say
that $\Omega$ satisfies the Harnack Chain condition if there is a uniform constant $C$ such that
for every $\rho >0,\, \Lambda\geq 1$, and every pair of points
$X,X' \in \Omega$ with $\delta(X),\,\delta(X') \geq\rho$ and $|X-X'|<\Lambda\,\rho$, there is a chain of
open balls
$B_1,\dots,B_N \subset \Omega$, $N\leq C(\Lambda)$,
with $X\in B_1,\, X'\in B_N,$ $B_k\cap B_{k+1}\neq \emptyset$
and $C^{-1}\diam (B_k) \leq \dist (B_k,\partial\Omega)\leq C\diam (B_k).$  The chain of balls is called
a ``Harnack Chain''.
\end{definition}

\begin{definition}({\bf NTA and 2-sided NTA}). \label{def1.nta} Again following \cite{JK}, we say that a
domain $\Omega\subset \ree$ is NTA (``Non-tangentially accessible") if it satisfies the
Harnack Chain condition, and if both $\Omega$ and
$\Omega_{\rm ext}:= \ree\setminus \overline{\Omega}$ satisfy the Corkscrew condition (for domains). If $\om$ and $\Omega_{\rm ext}$ both NTA domains then we say $\om$ is a 2-sided NTA domain. Note that in either case $\om$ is a UR domain.
\end{definition}

\begin{definition}({\bf Chord arc domain and 2-sided Chord arc domain}). We say that a domain 
$\om \subset \ree$ is a chord arc (resp. 2-sided chord arc) domain, if $\om$ is an NTA 
(resp. 2-sided NTA) domain, and $\pom$ is n-dimensional ADR.
\end{definition}

\begin{definition}({\bf Measure theoretic boundary}).\label{MTbdry} 
Given $\om \subset \ree$, a
set of locally finite perimeter\footnote{We note that if $\pom$ is ADR
(in particular, if it is UR), then $\om$ has locally finite perimeter, by the criterion in 
\cite[Theorem 1, p. 222]{EG}.}, we say that $x \in \partial_*\om$, the measure theoretic boundary of $\om$, if
\begin{equation}\label{MT1}
\limsup_{r \to 0} \frac{|B(x,r) \cap \om|}{r^{n+1}} > 0
\end{equation}
and
\begin{equation}\label{MT2}
\limsup_{r \to 0} \frac{|B(x,r) \cap \om^c|}{r^{n+1}} > 0,
\end{equation}
where $|A|$ is the Lebesgue measure of $A$.
Given a domain $\om\subseteq\ree$ we say that the measure theoretic boundary has full measure if $H^{n} (\pom \setminus \partial_*\om) = 0$. One should note that if $\om \subset \ree$ is a set of locally finite perimeter then the measure theoretic boundary and the reduced boundary differ by a set of $H^n$ measure zero, so it then follows that the measure theoretic boundary has full measure if and only the reduced boundary has full measure (see \cite[Section 5.8]{EG}).  
\end{definition}

\begin{definition}\label{NTapp}({\bf Nontangential approach region and maximal function}). Fix $\alpha > 0$ and let $\om$ be a domain 
then for $x \in \pom$ we define the nontangential approach region (or ``cone")
\begin{equation}\label{NTapp1}
\Gamma(x) = \Gamma_\alpha(x) = \{Y \in \om : |Y - x| < (1 + \alpha)\delta(Y)\}. 
\end{equation}
We also define the nontangential maximal function for $u: \om \to \re$
\begin{equation}\label{NTapp2}
\mathcal{N}u(x) = \mathcal{N}_\alpha u(x) = \sup_{Y \in \Gamma_\alpha(x)}|u(Y)|, \quad x \in \pom.
\end{equation}
We make the convention that $\mathcal{N}u(x) = 0$ when $\Gamma_\alpha(x)=\emptyset$.
\end{definition}
\begin{definition}\label{defvmo} {\bf (VMO and VMO$_{loc}$)}.
Let $E \subset \ree$ be $n$-dimensional ADR, and let $\sigma:=H^n|_E$ as above. 
We denote by $VMO(d\sigma)= VMO(E,d\sigma)$ 
the closure of the set of bounded uniformly continuous functions defined on 
$E$ in $BMO(E,d\sigma$).    
We say that $f \in VMO_{loc}(d\sigma)$ if 
\begin{equation}\label{eq1.26}
\lim_{r\to 0}\, \sup_{x\in K} \fint_{\Delta(x,r)} \left|f(x) - \fint_{\Delta(x,r)} f\right|\, d\sigma = 0\,,
\end{equation}
for every compact $K\subset E$.  Of course,  it is well known that \eqref{eq1.26}
holds with $E$ in place of $K$ in the supremum, if and only if $f \in VMO(E, d\sigma)$.  Thus, 
$VMO$ and $VMO_{loc}$ are distinct only for unbounded $E$.
\end{definition}

We now record some estimates and some known results that will be important in the proofs of the theorems.

\subsection{Preliminary Estimates and Observations}

In the sequel, we will sometimes assume more on $\om$ and ${\Omega_{\rm ext}}$, however $\om$ and ${\Omega_{\rm ext}}$ will always be UR domains. In addition, $k_1$ will be the Poisson kernel for the domain $\om$ and $k_2$ will be the Poisson kernel for ${\Omega_{\rm ext}}$.

Suppose that $\log k_1^{X_1} \in VMO_{loc}(d \sigma)$. 
Let $\hm=\hm^{X_1}$ and $k=k^{X_1}$ 
be the harmonic measure and Poisson Kernel for $\om$ with pole at ${X_1}$ (resp.). Then by the definition of $VMO_{loc}(d\sigma)$, 
for each $B_0 := B(x_0,r_0)$, with  $x_0\in \pom$, and $r_0>0$, and for each $\eta > 0$,
there exists $r_1 >0$ such that if $x \in \pom\cap B_0$, 
and $s \in (0,r_1)$, then
\begin{equation}\label{eq1}
\fint_{\Delta(x,s)} \left| \log k(z) - \fint_{\Delta(x,s)} \log k (y) \, d\sigma(y)\right| \, d\sigma(z) \le \eta.
\end{equation}
Since 
$\pom$ is ADR, it follows that $k$ and $k^{-1}$ belong to
$RH_{q,loc}(d\sigma)$ (``local Reverse H\"older-$q$"), for every $q<\infty$;  i.e., given $B_0=B(x_0,r_0)$,
for every $q \in (1,\infty)$,  there exists $C_{q,B_0}$ 
such that  for any
surface ball $\Delta =\Delta(x,r)$, with $x\in  B_0\cap\pom$, and $r\leq r_0 $,
\begin{equation}\label{eq2}
\left(\fint_\Delta k^q \, d\sigma \right)^{1/q} \le C_{q,B_0} \fint_{\Delta} k \, d\sigma.
\end{equation}
and 
\begin{equation}\label{eq3}
\left(\fint_{\Delta} k^{-q} \, d\sigma \right)^{1/q} \le C_{q,B_0} \fint_{\Delta} k^{-1} \, d\sigma.
\end{equation}
See, e.g., \cite[Theorem 2.1, p. 332]{KT3}, or the references cited there.
As a consequence of \eqref{eq2} and \eqref{eq3} (holding for all $q <\infty$) we have the following Lemma.

\begin{lemma}({\bf\cite{KT3} Corollary 2.4})\label{Lemma 1} Suppose that
$\log k \in VMO_{loc}(d\sigma)$.
Fix $B_0:= B(x_0,r_0)$, with $x_0\in \pom$ and $r_0>0$. 
Then for all $\beta > 0$, $w\in B_0\cap\pom$,  $s\leq r_0$, and $E \subset \Delta(w,s)=:\Delta\,$, 
\begin{equation}\label{eq4}
C_{\beta,B_0}^{-1}\left(\frac{\sigma(E)}{\sigma(\Delta)} \right)^{1+\beta} \le \frac{\hm(E)}{\hm(\Delta)} \le 
C_{\beta,B_0} \left(\frac{\sigma(E)}{\sigma(\Delta)} \right)^{1-\beta}.
\end{equation}
\end{lemma}
The reverse H\"older estimate \eqref{eq2} for some fixed $q>1$
(i.e., the $A_\infty$ property) yields an exponential reverse Jensen inequality, so that
for any $\Delta$ as in  \eqref{eq2},
\begin{equation}\label{eq6}
e^{\,\fint_{\Delta} \log k \, d \sigma} \approx \fint_{\Delta} k \, d \sigma = \frac{\hm(\Delta)}{\sigma(\Delta)},
\end{equation}
with implicit constants that may depend on $B_0$.
See \cite[Theorem 2.15, p. 405]{GR} for a proof of \eqref{eq6}. 
For the connection between $A_\infty$ and BMO, see \cite[Corollary 2.19, p. 409]{GR}. 

We shall require also the following.  

\begin{lemma}\label{Lemma 2}   Let
$B_0:= B(x_0,r_0)$, with $x_0\in\pom$ and $r_0>0$.
Given $\epsilon \in (0,1)$, let $r_1 >0$ be such that \eqref{eq1} 
holds for all $x \in \pom\cap B_0$  
and $s \in (0, r_1)$, with $\eta = \epsilon$. 
Let $p \in (1,\infty)$ and suppose that $Mr \in (0, r_1)$.   If
we set $\Delta^* : = \Delta(x, Mr)$, and $\displaystyle a: = a_{x,Mr} = e^{\,\fint_{\Delta^*} \log k \, d \sigma}$, 
one then has
\begin{equation}\label{eq7}
\left(\fint_{\Delta^*} \left|1 - \frac{k}{a} \right|^p \, d \sigma \right)^{1/p} \le C_1 \epsilon^b,
\end{equation}
with $b = b(p) =(4p)^{-1}$, and for some
$C_1 = C_1(p,  B_0)$, uniformly for $x\in \pom\cap B_0$  and $r<r_1/M$.
\end{lemma}

\noindent{\it Remark}.  Applying \eqref{eq6} with $\Delta^*$ in place of $\Delta$,
we note for future reference that
\begin{equation}\label{eq6a}
a \approx  \frac{\hm(\Delta(x,Mr))}{\sigma(\Delta(x,Mr))},
\end{equation}

\begin{proof}
Set $A := \{z \in {\Delta^*}: |\log k(z) - \fint_{\Delta^*} \log k(y) \, d\sigma(y)| > \sqrt{\epsilon}\}$, so that 
by Chebyshev's inequality one has
\begin{equation}\label{eq8}
\sigma(A) \le \sqrt{\epsilon} \sigma(\Delta^*).
\end{equation}
Set $F := \Delta^* \setminus A$.   Then by definition of $a$, we have for $z \in F$,
\begin{equation}\label{eq9}
\left| \log k(z) - \log a \right| \le \sqrt{\epsilon}.
\end{equation}
Exponentiating \eqref{eq9} we obtain
\begin{equation}\label{eq10}
e^{-\sqrt{\epsilon}} \le \frac{k(z)}{a} \le e^{\sqrt{\epsilon}},
\end{equation}
so that by Taylor's Theorem one has
\begin{equation}\label{eq11}
\left|1 - \frac{k(z)}{a}\right| \le C\sqrt{\epsilon}\,,\qquad \forall\, z\in F\,.
\end{equation}
Then we have
\begin{multline}\label{eq12}
\left(\fint_{\Delta^*} \left|1 - \frac{k}{a} \right|^p \, 
d \sigma \right)^{1/p} \\[4pt]
\le \, \left(\frac{1}{\sigma(\Delta^*)} \int_{F} \left|1 - \frac{k}{a} \right|^p \, d \sigma\right)^{1/p} + 
\left(\frac{1}{\sigma(\Delta^*)} \int_{A} \left|1 - \frac{k}{a} \right|^p \, d \sigma \right)^{1/p} \\[4pt]
\lesssim \sqrt{\epsilon} + \left(\frac{1}{\sigma(\Delta^*)} \int_{A} \left|1 - \frac{k}{a} \right|^p \, d \sigma \right)^{1/p},
\end{multline}
Using Minkowski's inequality, \eqref{eq8}, \eqref{eq6a}, 
and that $k \in RH_{2p}$ one has 
\begin{equation}\label{eq13}
\begin{split}
\left(\frac{1}{\sigma(\Delta^*)} \int_{A} \left|1 - \frac{k}{a} \right|^p \, d \sigma \right)^{1/p} &\le \left(\frac{\sigma(A)}{\sigma(\Delta^*)} \right)^{1/p} + \frac{1}{a} \left(\frac{1}{\sigma(\Delta^*)} \int_A k^p \right)^{1/p}\\
&\le \epsilon^{\frac{1}{2p}} + \left(\frac{\sigma(A)}{\sigma(\Delta^*)}\right)^{\tfrac{1}{2p}} \frac{1}{a}\left(\frac{1}{\sigma(\Delta^*)} \int_A k^{2p}  \, d \sigma\right)^{1/2p}\\
&\le \epsilon^{\frac{1}{2p}} + \epsilon^{\frac{1}{4p}}\frac{1}{a}  \fint_{\Delta^*} k \, d \sigma \,
\lesssim \,\epsilon^{\frac{1}{ 4p}}\,,
\end{split}
\end{equation}
where the implicit constants may depend on $B_0$.
Then, since $p > 1$, putting \eqref{eq12} and \eqref{eq13} together we obtain \eqref{eq7},
with $C_1$ depending on $B_0$ and $p$. 
\end{proof}
Since $\pom$ is UR, 
for all $p \in (1,\infty)$ we have
\begin{equation}\label{eq15}
\lVert \mathcal{N}(\nabla \mathcal{S} f)\rVert_p  \le C \lVert f \rVert_p,
\end{equation}
where $C$ depends on the UR character of $\pom$, dimension, $p$, and the aperture of the cones defining
$\mathcal{N}$.   Estimate \eqref{eq15} is essentially proved in \cite{DS1} (bounds for the non-tangential maximal
function of $\nabla \mathcal{S} f$ follow from uniform bounds for the truncated singular integrals, plus a standard 
Cotlar Lemma argument;  the details may be found in \cite[Proposition 3.20]{HMT}.)

In addition, we have the following result proved in \cite{HMT}.
\begin{lemma}\label{hmtlemma}(\cite[Proposition 3.30]{HMT}), 
If $\om$ is a UR domain (recall Definition \ref{URdom}), 
then for  a.e. $x \in \pom$, and for all $f \in L^p(d\sigma)$, $1<p<\infty$,
\begin{equation}\label{eq16}
\lim_{\substack{Z \to x \\ Z \in \Gamma^-(x)}}\nabla \mathcal{S} f (Z) = -\frac{1}{2}\nu(x)f(x) + Tf(x)\,,
\end{equation}
and 
\begin{equation}\label{eq30}
\lim_{\substack{Z \to x \\ Z \in \Gamma^+(x)}}\nabla \mathcal{S} f (Z) = \frac{1}{2}\nu(x)f(x) + Tf(x)\,.
\end{equation}
where $T$ is a principal value singular integral operator, $\Gamma^+(x)$ is the cone at 
$x$ relative to $\om$, $\Gamma^-(x)$ is the cone at $x$ relative to ${\Omega_{\rm ext}}$, and $\nu$ is the 
outer normal to $\om$. 
\end{lemma}

\begin{remark} We have taken our fundamental solution to be positive, so 
for that reason there are some changes in sign in both \eqref{eq16} and \eqref{eq30} 
as compared to the formulation in \cite{HMT}. \end{remark}

\begin{remark}  Recall that by Definition \ref{URdom}, we have assumed in particular that
$\sigma(\pom\setminus \pom_*) = 0$, and therefore the measure theoretic outer unit normal
$\nu$ exists a.e. on $\pom$ (see \cite[Chapter 5]{EG}).
\end{remark}

We recall now some fundamental estimates relating harmonic measure and the Green function.
\begin{lemma}({\bf \cite{JK}})\label{CFMS} Let $\om \subset \ree$ be an NTA domain and let $G(X,Y)$ denote the Green function of $\om$. Let $x \in \pom$, $R \in (0, \diam \pom)$, 
and let $Y_R$ be a corkscrew point for $B(x,R)$.   
If $X \in \om \setminus B(x,2R)$, then
\begin{equation}\label{CFMS1}
G(X,Y_R) \approx \frac{\hm^X(\Delta(x, R))}{R^{n-1}}.
\end{equation}
Moreover,
if $g(Y)$ is the Green function with pole at infinity, and $\hm$ is harmonic measure associated to $g$ (see \cite{KT4}) then 
\begin{equation}\label{CFMS2}
G(Y_R) \approx \frac{\hm(\Delta(x, R))}{R^{n-1}}.
\end{equation}
The implicit constants in \eqref{CFMS1}-\eqref{CFMS2} depend only on dimension and the NTA constants.
\end{lemma}

\noindent{\it Remark.}  In the case that the pole is finite (estimate \eqref{CFMS1}), 
this result was proved in \cite{JK};  the result for pole at infinity follows from the finite case,
plus the construction in \cite{KT2}.
Estimate \eqref{CFMS1} has been extended to general divergence form elliptic operators in \cite{CFMS}.

\begin{lemma}(\cite{JK})\label{Carlest} Let $\om$ be an NTA domain and suppose that $u$ is harmonic in $\om$ and vanishes continuously on $\Delta(x,2r)$ then 
\begin{equation}\label{Carlest1}
\sup_{X \in B(x,r) \cap \om} u(X) \lesssim u(Y_{x,r})
\end{equation}
where $Y_{x,r}$ is a corkscrew point for $\Delta(x,r)$.  The implicit constants depend only on the NTA constants
and dimension.
\end{lemma}

We also recall a result proved by Bourgain about harmonic measure on domains with ADR boundary.

\begin{lemma}(\cite{Bo})\label{bourgainsLemma} Let $\om \subset \ree$ be an open set 
with $n$-dimensional ADR boundary $\pom$. Then there exists 
constants $C_0>2$, and $c_1 > 0$, depending on dimension and ADR such 
that for all $x_1 \in \pom$ and $R \in (0,\diam(\pom))$, if $Y \in \om \cap B(x_1, R)$ then
\begin{equation}\label{bourgainest}
\hm^Y(B(x_1,C_0R)) \ge c_1.
\end{equation}
\end{lemma}

\section{Proofs Theorems \ref{T1} and \ref{T2}}

\begin{proof}[Proof of Theorem \ref{T1}]
Fix $B_0=B(x_0,r_0)$, and $B_0^\star:=B(x_0, r_0^\star)$,
with $x_0\in\pom$,  $r_0>0$, and $r_0^\star = 100C_0(|X_1-x_0| + r_0 +1)$, 
where $C_0$ is the constant
in Lemma \ref{bourgainsLemma}.  Fix also
$\epsilon \in (0, 1)$, and
let $r_1 >0$ be such that \eqref{eq1} holds with $\eta = \epsilon$, 
 for $x \in \pom\cap B_0^\star$ and $s \in (0,r_1)$. 
 Without loss of generality we take $r_1 < \min\{\frac{\delta({X_1})}{50}, 1,r_0\}$. 
 Let $M := 50 \epsilon^{-\frac{1}{8n}}$, 
 and  suppose that $Mr \in (0,r_1)$.  We now fix $x \in B_0\cap\pom$, and for any points $y,z \in \Delta(x,r)$,
 let $y^*,z^*$ denote  arbitrary points in
 $\Gamma^-(y) \cap B(y,r/2)$ and in $\Gamma^-(z) \cap B(z,r/2)$, respectively. Setting 
 $$\Delta := \Delta(x,r)\,,\qquad \Delta^* := \Delta(x,Mr)\,,$$ 
 we shall first prove that for any such $y,z,y^*,z^*$,
\begin{equation}\label{eq17}
\left( \fint_\Delta \left| \nabla \mathcal{S} 1_{\Delta^*}(z^\ast) - \fint_\Delta \nabla \mathcal{S} 1_{\Delta^*}(y^*) \, d\sigma(y) \right|^2 \, d \sigma(z) \right)^\frac{1}{2} \le C \epsilon^\gamma,
\end{equation}
where $C = C(n,UR,B_0, B_0^\star, \delta(X_1),|X_1-x_0|)$, and 
where $\gamma$ is a positive constant, depending
only on dimension.  For the sake of notational convenience, in the sequel,
we shall often allow generic 
and implicit constants
to depend upon these parameters, without explicitly making note of such dependence.
As before, set $\displaystyle a: = a_{x,Mr} = e^{\,\fint_{\Delta^*} \log k \, d \sigma}$ and write
\begin{equation}\label{eq18}
1_{\Delta^*} = \left[\left(1 - \frac{k}{a}\right)1_{\Delta^*}\right] + \left[\frac{k}{a}\right] - \left[\left(\frac{k}{a}\right)1_{(\Delta*)^c}\right] .
\end{equation}
Using \eqref{eq18} we have that the left hand side of \eqref{eq17} is bounded by the sum of three terms $\RNum{1}, \RNum{2}$ and $\RNum{3}$ where
\begin{equation}\label{eq19}
\RNum{1} = \left( \fint_\Delta \left| \nabla \mathcal{S} \left[\left(1 - \frac{k}{a}\right)1_{\Delta^*}\right] (z^\ast) - \fint_\Delta \nabla \mathcal{S} \left[\left(1 - \frac{k}{a}\right)1_{\Delta^*}\right](y^*) \, d\sigma(y) \right|^2 \, d \sigma(z) \right)^\frac{1}{2},
\end{equation}
\begin{equation}\label{eq20}
\RNum{2} = \left( \fint_\Delta \left| \nabla \mathcal{S} \left[\frac{k}{a}\right] (z^\ast) - \fint_\Delta \nabla \mathcal{S} \left[\frac{k}{a}\right](y^*) \, d\sigma(y) \right|^2 \, d \sigma(z) \right)^\frac{1}{2},
\end{equation}
and
\begin{equation}\label{eq21}
\RNum{3} = \left( \fint_\Delta \left| \nabla \mathcal{S} \left[\left(\frac{k}{a}\right)1_{(\Delta*)^c}\right] (z^\ast) - \fint_\Delta \nabla \mathcal{S} \left[\left(\frac{k}{a}\right)1_{(\Delta*)^c}\right](y^*) \, d\sigma(y) \right|^2 \, d \sigma(z) \right)^\frac{1}{2}.
\end{equation}
We begin by estimating $\RNum{1}$.  By \eqref{eq15} and Lemma \ref{Lemma 2} with $p=2$, we have
\begin{multline}\label{eq23}
\RNum{1} \le  2 \left(\fint_\Delta \left| \mathcal{N} 
\left( \nabla \mathcal{S} \left[\left(1 - \frac{k}{a}\right)
1_{\Delta^*}\right]\right) \right|^2 \, d \sigma \right)^{\frac{1}{2}} \\*
\lesssim \,
M^{\frac{n}{2}} \left(\fint_{\Delta^\ast} 
\left| 
1 - \frac{k}{a} 
\right|^2 \, d \sigma \right)^{\frac{1}{2}}\,
\lesssim \, M^{\frac{n}{2}}\epsilon^{\frac{1}{8}} \,\lesssim\, \epsilon^{\frac{1}{16}}.
\end{multline}
Now for $\RNum{2}$, we recall that $k = k_1^{X_1}$ is harmonic measure for $\om$ with pole at $X_1$.  Moreover,
$\mathcal{E}(\cdot - z^*)$ and $\mathcal{E}(\cdot - y^*)$ are 
harmonic in $\om$ since $z^*, y^* \in {\Omega_{\rm ext}}$, and decay to 0 at infinity, and are therefore equal 
to their respective Poisson integrals in $\om$.   Consequently,
\begin{equation}\label{eq24}
\RNum{2} 
= \frac{1}{a}  \left( \fint_\Delta  \fint_\Delta \left| \nabla \mathcal{E}({X_1} - z^*)  - \nabla \mathcal{E}({X_1} - y^*) \, d\sigma(y) \right|^2 \, d \sigma(z) 
\right)^\frac{1}{2}\,.
\end{equation}
Recall that by definition, $Mr +r_0< 2r_0 \ll r_0^\star$, and that $x\in B_0\cap\pom$,
hence $B(x,Mr) \subset 2B_0\subset B_0^\star$.   Thus, we may apply Lemma \ref{Lemma 1},
with $r_0^\star, \,B_0^\star$ in place of $r_0,\, B_0$, with $w=x_0$ and $s = r_0^\star$
(thus, $B(w,s)= B_0^\star$, $\Delta=\Delta_0^\star  =B_0^\star\cap\pom$), and with 
$E=\Delta(x,Mr)$, to deduce that
\begin{equation}\label{eq2.8a}
\frac{\hm(\Delta_0^\star)}{\hm(\Delta(x,Mr)} \lesssim 
\left(\frac{r_0^\star}{Mr}\right)^{n+\tau}\,,\qquad \forall \, \tau>0\,,
\end{equation}
where the implicit constant of course depends upon $\tau$ and $B_0^\star$.

We note that by Lemma \ref{bourgainsLemma}, and the definition of
$r_0^\star$, there is a uniform constant $c>0$ such that
\begin{equation}\label{eq2.8}\hm(\Delta^\star_0) \ge c\,.
\end{equation}
We further note that, since $y^*,z^* \in B(x,2r)$, 
\begin{equation*}
\left| \nabla \mathcal{E}({X_1} - z^*)  - \nabla \mathcal{E}({X_1} - y^*)   \right| 
\lesssim \frac{r}{\delta(X_1)^{n+1}}.
\end{equation*}
Then continuing \eqref{eq24}, we have, using \eqref{eq6a}, 
\eqref{eq2.8a} (with $\tau=1/2$), and  \eqref{eq2.8}, 
\begin{multline}\label{eq38}
\RNum{2} 
\lesssim \frac{1}{a} \frac{r}{\delta(X_1)^{n+1}} 
\approx \frac{\sigma(\Delta(x,Mr))}{\hm(\Delta(x,Mr))} \frac{r}{\delta(X_1)^{n+1}} \\[4pt]
=\frac{\sigma(\Delta(x,Mr))}{\hm(\Delta^{\star}_0)}\frac{\hm(\Delta^{\star}_0)}{\hm(\Delta(x,Mr))} 
\frac{r}{\delta(X_1)^{n+1}} \\[4pt]
\lesssim \,(Mr)^n\left(\frac{r_0^\star}{Mr} \right)^{n+\frac{1}{2}}\frac{r}{\delta(X_1)^{n+1}} 
\,\lesssim \, 
M^{-\frac{1}{2}}r^{\frac{1}{2}} 
\,\lesssim \frac1M \approx \epsilon^{\frac{1}{8n}}\,, 
\end{multline}
since  $r<r_1/M \ll \delta(X_1)/M$, where we remind the reader that in this part of the argument,
we allow implicit constants to depend upon $\delta(X_1)$, and on the various parameters involved 
in the definition of $r_0^\star$.

For $\RNum{3}$ we use basic Calder\'{o}n-Zygmund type estimates as follows.  Let 
$$\Delta_j' := \Delta(x, 2^jr)\,,\qquad A'_j := \Delta_j' \setminus \Delta_{j-1}'\,,$$ 
so that
\begin{multline}\label{eq25}
\RNum{3} = \\ 
 \left( \fint_\Delta \left| \fint_\Delta 
 \left(\nabla \mathcal{S} \left[\left(\frac{k}{a}\right)
 1_{(\Delta*)^c}\right] (z^\ast) -  \nabla \mathcal{S} \left[\left(\frac{k}{a}\right)
 1_{(\Delta*)^c}\right](y^*)\right) \, d\sigma(y) \right|^2 \, d \sigma(z) \right)^\frac{1}{2} \\
 = \left( \fint_\Delta \left| \fint_\Delta \int_{\pom\setminus \Delta^*} \right[\nabla 
\mathcal{E}(z^* - w) -\nabla \mathcal{E}(y^*-w)\left]\frac{k(w)}{a} \, d\sigma(w)  
\, d\sigma(y) \right|^2 \, d \sigma(z) \right)^\frac{1}{2}  \\
\le \sum_{j: 2^j \ge M} \left( \fint_\Delta \left[\fint_\Delta   
 \int_{A'_j} \left|\nabla \mathcal{E}(z^* - w) -\nabla \mathcal{E}(y^*-w)\right|\frac{k(w)}{a} \, d\sigma(w) 
  \, d\sigma(y) \right]^2 \, d \sigma(z) \right)^\frac{1}{2} \\
 \lesssim  \sum_{j: 2^j \ge M} \left( \fint_\Delta \left[ \fint_\Delta  \int_{A'_j} 
\frac{r}{(2^jr)^{n+1}}\frac{k(w)}{a} \, d\sigma(w)  \, d\sigma(y) \right]^2 \, d \sigma(z) \right)^\frac{1}{2}\\
=\,\, \sum_{j\in J_1}\, ... \,\, +\,\, \sum_{j \in J_2}\,  ...\, =: III_1 \, +\, III_2\,,
\end{multline}
where
$$J_1:= \left\{j:\, M\leq 2^j \leq r_0^\star /r\right\}\,, \quad J_2:= \left\{j: \,2^j  > r_0^\star/r\right\}\,.$$
For $j\in J_1$, we may apply Lemma \ref{Lemma 1}, with $B_0^\star$ in place of
$B_0$, and with $\Delta = \Delta_j'$ and $E = \Delta(x,Mr)$, to obtain
\begin{equation}\label{eq2.12}
\frac{\hm(\Delta_j')}{\hm(\Delta(x,Mr)} \lesssim  \left(\frac{2^{j}}{M}\right)^{n+\beta n} 
\end{equation}
We then have
\begin{multline*}
III_1 \,\lesssim\,\,
 \sum_{j\in J_1} \, \frac{1}{2^j}
\left( \fint_\Delta \left[ \fint_\Delta  \fint_{\Delta_j'} \frac{k(w)}{a} \, 
d\sigma(w)  \, d\sigma(y) \right]^2 \, d \sigma(z) \right)^\frac{1}{2} \\
 \lesssim \, \sum_{j\in J_1} \, \frac{1}{2^j} \frac{1}{a}\frac{\hm(\Delta_j')}{\sigma(\Delta_j')}\,
\approx \, \sum_{j\in J_1} \, \frac{1}{2^j} \frac{\sigma(\Delta(x,Mr))}{\sigma(\Delta_j')} \frac{\hm(\Delta_j')}{\hm(\Delta(x,Mr)}\\
\lesssim\, \sum_{j: 2^j \geq M} 2^{-j}  \left(\frac{M}{2^j}\right)^n  \left(\frac{2^{j}}{M}\right)^{n+1/2} 
\,\lesssim \,\frac{1}{M} \,\lesssim \,\epsilon^{\frac{1}{8n}}\,,
\end{multline*}
where in the middle line we have used  \eqref{eq6a}, and
in the last line, the ADR property and \eqref{eq2.12} with $\beta =1/(2n)$.
Since $\hm$ is a probablility measure, we also have that
\begin{multline*}III_2 \,\lesssim \, \sum_{j: 2^j r > r_0^\star}r \, (2^jr)^{-n-1}\, a^{-1} \,
\approx \, r \,\left(\frac1{r_0^\star}\right)^{n+1}\frac{\sigma(\Delta(x,Mr))}{\hm(\Delta(x,Mr))}\\
=\,  r \,\left(\frac1{r_0^\star}\right)^{n+1}\frac{\sigma(\Delta(x,Mr))}{\hm(\Delta_0^*)}
\frac{\hm(\Delta_0^*)}{\hm(\Delta(x,Mr))}\,\lesssim \, \
 r \,\left(\frac1{r_0^\star}\right)^{n+1} (Mr)^n \left(\frac{r_0^\star}{Mr}\right)^{n+1/2}\\
 =\, \left(\frac{r}{r_0^\star}\right)^{1/2} M^{-1/2} \lesssim M^{-1} \approx \epsilon^{\frac1{8n}}\,,
\end{multline*}
where in the middle line we have used the ADR property, \eqref{eq2.8}, and
\eqref{eq2.8a} (with $\tau = 1/2$), and in the last line that $Mr \leq r_1\ll r_0^\star$.
Combining the estimates for $I$, $II$, $III_1$ and $III_2$, we obtain \eqref{eq17}
with $\gamma = 1/(8n)$
(or with $\gamma =1/16$, if $n=1$).  

With  \eqref{eq17} in hand, we continue with the proof of Theorem \ref{T1}.  Setting
\begin{equation*}
\nt^- f(x) := \lim_{\substack{Z \to x \\ Z \in \Gamma^-(x)}} \nabla \mathcal{S}f(Z) \, ,
\end{equation*}
since the limit exists for a.e. $x\in \pom$ (see Lemma \ref{hmtlemma}),
we may now use \eqref{eq15},  \eqref{eq17}, and 
dominated convergence to obtain  
\begin{equation}\label{eq28}
\left( \fint_\Delta \left| \nt^- 1_{\Delta^*}(z) - \fint_\Delta \nt^- 1_{\Delta^*}(y) \, d\sigma(y) \right|^2 \, d \sigma(z) \right)^\frac{1}{2} \le C \epsilon^\gamma\,,
\end{equation}
for $x \in \pom\cap B_0$ 
and $r < r_1/M$.
In addition since $\log k_2^{X_2} \in VMO(d\sigma)$ the same analysis shows that \eqref{eq28} holds for $\nt^- 1_{\Delta^*}$ replaced with
\begin{equation}\label{eq29}
\nt^+ 1_{\Delta^*} : = \lim_{\substack{Z \to x \\ Z \in \Gamma^+(x)}}\nabla\mathcal{S}1_{\Delta^*}(Z).
\end{equation}
By \eqref{eq16} and \eqref{eq30}
\begin{equation}\label{eq31}
\nu(x) 1_{\Delta^*}(x) = \lim_{\substack{Z \to x \\ Z \in \Gamma^+(x)}}\nabla\mathcal{S}1_{\Delta^*}(Z) - \lim_{\substack{Z \to x \\ Z \in \Gamma^-(x)}}\nabla\mathcal{S}1_{\Delta^*}(Z)\,.
\end{equation}
Thus, since $\Delta \subset \Delta^*$,
by \eqref{eq28} and its analogue for $\mathcal{S}^+$, we obtain
\begin{equation}\label{eq33}
\left( \fint_\Delta \left| \nu(z) - \fint_\Delta \nu(y) \, d\sigma(y) \right|^2 \, d \sigma(z) \right)^\frac{1}{2} \le C\epsilon^\gamma,
\end{equation}
for $x \in \pom\cap B_0$ and $0< r< M^{-1}r_1(\epsilon) \approx \epsilon^{1/(8n)} r_1(\epsilon)$. 
Hence, $\nu \in VMO_{loc}(d\sigma)$.
\end{proof}

\begin{proof}[Proof of Theorem \ref{T2}]
The proof is nearly the same as that of Theorem \ref{T1}, with a few minor differences.
Recall that in contrast to the situation for Theorem \ref{T1}, we now impose the stronger assumption that $\om$ is a chord arc domain and that $\log k_1 \in VMO$ (globally). In this setting,  Lemma \ref{CFMS} and Lemma \ref{Carlest} hold,
and harmonic measure is doubling.
Moreover,  \eqref{eq1}, Lemma \ref{Lemma 1} and Lemma \ref{Lemma 2} hold globally
(i.e., not localized to a ball $B_0$). Following the proof of Theorem \ref{T1},  we proceed as follows.

As before, we fix $\epsilon \in (0,1)$, and 
let $r_1>0$ be such that \eqref{eq1} holds with $\eta = \epsilon$, 
 for $x \in \pom$ and $s \in (0,r_1)$. 
We take $r_1 < \min\{\frac{\delta({X_1})}{50}, 1\}$, let $M := 50 \epsilon^{-\frac{1}{8n}}$, 
 and  suppose that $Mr \in (0,r_1)$.  
 Given $x \in \pom$, we set 
 $\Delta := \Delta(x,r)$ and $ \Delta^* := \Delta(x,Mr),$ and for $y,z\in\Delta$, we
 let $y^*,z^*$ denote arbitrary points in
 $\Gamma^-(y) \cap B(y,r/2)$, and  $\Gamma^-(z) \cap B(z,r/2)$, respectively. 
Once again, we seek to prove \eqref{eq17},
but now with $C = C(n,UR)$.  

We break up $1_{\Delta^*}$ as in \eqref{eq18} and estimate the left hand side of \eqref{eq17} by 
the same three  terms, $\RNum{1}, \RNum{2}$ and $\RNum{3}$. We estimate term  $\RNum{1}$ 
exactly as before.  

Next, we will show that 
\begin{equation}\label{eq34}
\RNum{2} = \left( \fint_\Delta \left| \nabla \mathcal{S} \left[\frac{k}{a}\right] (z^\ast) - \fint_\Delta \nabla \mathcal{S} \left[\frac{k}{a}\right](y^*) \, d\sigma(y) \right|^2 \, d \sigma(z) \right)^\frac{1}{2} = 0,
\end{equation}
for $k = k_1$, the Poisson kernel at infinity.
Fix  $x_0\in \pom$ and $R \gg r$, and let $\varphi_R(X) \in C_0^\infty$ be a postivite smooth cutoff function on $B(x_0,R)$ 
such that $\supp \varphi \subset B(x_0,R),$  
$|\nabla \varphi| \lesssim \frac{1}{R},\, |\nabla^2 \varphi| \lesssim \frac{1}{R^2}$, 
and $\varphi \equiv 1$ on $B(x_0,\tfrac{R}{2})$.  Let
\begin{multline*}
\RNum{2}_R := \\[4pt]
\left(\fint_\Delta \left|   \fint_\Delta \int_{\pom} \left[\nabla \mathcal{E}(w - z^*)  - \nabla \mathcal{E}(w - y^*)\right]\varphi_R(w) \left[\frac{k(w)}{a}\right] \, d\sigma(w) \, d\sigma(y) \right|^2 \, d \sigma(z)\right)^{\frac{1}{2}} 
\end{multline*}
If we set $A_R = B(x_0, R) \setminus B(x_0,\frac{R}{2})$, let
$g(X)$ be the Green function with pole at infinity,  and let $\mathcal{L} := \nabla\cdot\nabla$
denote the usual Laplacian in $\ree$, then by definition
\begin{multline*}
\RNum{2}_R^2 = \\*[4pt]
\frac{1}{a^2}\fint_\Delta \left| \fint_\Delta \iint_{\om} 
\mathcal{L}\Big( \left[\nabla \mathcal{E}(W - z^*)  - \nabla 
\mathcal{E}(W - y^*)\right]\varphi_R(W)\Big) g(W) \, dW \, d\sigma(y) \right|^2 \, d \sigma(z) \\*[4pt]
\lesssim \frac{1}{a^2}\left(\iint_{A_R}g(W) \frac{1}{R^{n+2}} \frac{r}{R} \, dW\right)^2.
\end{multline*}
Let $Y_R$ be a corkscrew point for $B(x_0, R)$, and set $\Delta_R = \pom \cap B(x_0, R)$.   Then using 
Lemma \ref{Carlest}, Lemma \ref{CFMS}  and \eqref{eq4} (without dependence on $B_0$) we have that
\begin{equation}
\begin{split}
\frac{1}{a^2}\left(\iint_{A_R}g(W) \frac{1}{R^{n+2}} \frac{r}{R} \, dW\right)^2 &\lesssim \frac{1}{a^2}\left(\frac{r}{R^{2}}\frac{1}{R^{n+1}} \iint_{B(x_0,R)}g(W)  \, dW\right)^2  \\
&\lesssim \frac{1}{a^2} \left( \frac{r}{R^2} g(Y_R) \right)^2 \\
&\lesssim \frac{1}{a^2} \left(\frac{r}{R} \frac{\hm(\Delta_R)}{\sigma(\Delta_R)} \right)^2 \\
&\lesssim \left(\frac{\sigma(\Delta(x_0, Mr))}{\sigma(\Delta_R)} \frac{\hm(\Delta_R)}{\hm(\Delta(x_0,Mr))}\right)^2 \frac{r^2}{R^2}\\
&\lesssim \frac{r}{R},
\end{split}
\end{equation}
Thus, for  $z^*,y^*$ as above, $\RNum{2}_R \to 0$ as $R \to \infty$, uniformly in the points 
$z^*$ and $y^*$ under consideration. By the Dominated Convergence Theorem, we see that 
\begin{equation}
\RNum{2} = \lim_{R \to \infty} \RNum{2}_R = 0.
\end{equation}
Finally, to handle term $\RNum{3}$, we proceed as before until, as in \eqref{eq25}, we obtain
\begin{equation}\label{eq37}
\RNum{3} \lesssim  \sum_{j: 2^j \ge M} \left( \fint_\Delta \left[ \fint_\Delta  \int_{A'_j} 
\frac{r}{(2^jr)^{n+1}}\frac{k(w)}{a} \, d\sigma(w)  \, d\sigma(y) \right]^2 \, d \sigma(z) \right)^\frac{1}{2}.
\end{equation}
In the present setting, we can estimate everything as we did for $\RNum{3}_1$ in the proof of Theorem \ref{T1} (since we have no dependence on $B^\star_0$ or $B_0$).  Exactly as before,
we find that
$$III\lesssim \epsilon^{\frac1{8n}}\,,$$
and the conclusion of Theorem \ref{T2} follows.
\end{proof}

\noindent{\bf Acknowledgements}.  We are grateful to Chema Martell for pointing out an error
in the original version of this manuscript, which could be fixed by working with local VMO spaces.

\end{document}